\documentclass{article}%
\usepackage{amsfonts}
\usepackage{graphicx}
\usepackage{amsmath}%
\usepackage{amssymb}
\providecommand{\U}[1]{\protect\rule{.1in}{.1in}}
\newtheorem{theorem}{Theorem}

\newtheorem{corollary}[theorem]{Corollary}

\newtheorem{definition}[theorem]{Definition}

\newtheorem{lemma}[theorem]{Lemma}
\newtheorem{notation}[theorem]{Notation}

\newenvironment{proof}[1][Proof]{\textbf{#1.} }{\ \rule{0.5em}{0.5em}}
\begin{document}

\title{From Lie Algebras to Lie Groups\\within Synthetic Differential Geometry:\\Weil Sprouts of \\Lie's Third Fundamental Theorem}
\author{Hirokazu NISHIMURA\\Institute of Mathematics\\University of Tsukuba\\Tsukuba, Ibaraki, 305-8571, JAPAN}
\maketitle

\begin{abstract}
Weil prolongations of a Lie group are naturally Lie groups. It is not known in
the theory of infinite-dimensional Lie groups how to construct a Lie group
with a given Lie algebra as its Lie algebra or whether there exists such a Lie
group at all. We will show in this paper how to construct some Weil
prolongations of this mythical Lie group from a given Lie algebra. We will do
so within our favorite framework of synthetic differential geometry.

\end{abstract}

\section{\label{s1}Introduction}

In the theory of finite-dimensional Lie groups, Lie's third fundamental
theorem is usually established via the Levi decomposition. The
Levi-Mal'\v{c}ev theorem asserts the existence of a Levi decomposition for any
finite-dimensional Lie algebra. That is to say, any finite-dimensional Lie
algebra is the semidirect product of a solvable Lie algebra $\mathfrak{m}%
$\ and a semisimple Lie algebra $\mathfrak{q}$. Since it is easy to establish
Lie's third fundamental theorem in both the solvable case and the semisimple
one, the desired Lie group is obtained as the semidirect product of the
established Lie groups with their respective Lie algebras $\mathfrak{m}$\ and
$\mathfrak{q}$, for which the reader is referred, e.g., to \S 3.15 of
\cite{va}.

This route to Lie's third fundamental theorem does not seem susceptible of any
meaningful infinite-dimensional generalization, and, as far as we know, Lie's
third fundamental theorem is not available in the theory of
infinite-dimensional Lie groups at present. In this sense ''a Lie group with a
given Lie algebra as its Lie algebra'' beyond the finite-dimensional realm is
\textit{mythical}. The principal objective in this paper is to show that some
Weil prolongations of this mythical Lie group are \textit{real} to our great
surprise, which can be regarded as \textit{Weil sprouts} of Lie's third
fundamental theorem in a sense. We will do so within our favorite framework of
synthetic differential geometry, for which the reader is referred to
\cite{la}. Our considerations shall be restricted to lower-dimensional cases
because of computational complexity.

The construction of this paper goes as follows. After giving some
preliminaries in the coming section, we will study some Weil prolongations of
a Lie group, which are again Lie groups, and their Lie algebras in
\S \ref{s3}. In the succeeding section we will establish a slight
generalization of the Baker-Campbell-Hausdorff formula discussed in the
concluding two sections of our previous paper \cite{ni}, which enables us to
endow some Weil prolongations of a given Lie algebra with Lie group
structures. The concluding two sections are devoted to showing that the
derived structures really make some Weil prolongations of a given Lie algebra
Lie groups of the desired Lie algebras.

\section{\label{s2}Preliminaries}

We assume the reader to be familiar with the first three chapters of \cite{la}
as well as the first five section of \cite{ni}.

The following theorem is due to Kock \cite{ko}.

\begin{theorem}
\label{t2.1}Let $\mathbb{E}$ be a Euclidean $\mathbb{R}$-module. Given
\ \ $f\in\mathbb{E}^{D_{n}}$, there exist unique $X_{0},X_{1},...,X_{n}%
\in\mathbb{E}$ such that
\[
f(d)=X_{0}+dX_{1}+...+d^{n}X_{n}%
\]
for any $d\in D_{n}$
\end{theorem}

A variant of the above theorem is

\begin{theorem}
\label{t2.2}Let $\mathbb{E}$ be a Euclidean $\mathbb{R}$-module. Given
\ \ $f\in\mathbb{E}^{D_{n}\times D}$, there exist unique $X_{0},X_{1}%
,...,X_{n},Y_{0},Y_{1},...,Y_{n}\in\mathbb{E}$ such that
\[
f(d,e)=\left(  X_{0}+eY_{0}\right)  +d\left(  X_{1}+eY_{1}\right)
+...+d^{n}\left(  X_{n}+eY_{n}\right)
\]
for any $d\in D_{n}$ and any $e\in D$.
\end{theorem}

\begin{proof}
The exponential law ensures that
\[
\mathbb{E}^{D_{n}\times D}=\left(  \mathbb{E}^{D_{n}}\right)  ^{D}%
\]
Thanks to Theorem \ref{t2.1}, we are allowed to think that
\[
\mathbb{E}^{D_{n}}=\mathbb{E}^{n+1}%
\]
Since $\mathbb{E}$ is Euclidean, $\mathbb{E}^{n+1}$\ is also Euclidean.
Therefore we have the desired result.
\end{proof}

We recall that

\begin{definition}
A Lie group is a group which is microlinear as a space.
\end{definition}

Similarly, we should be precise in saying a ''Lie algebra''

\begin{definition}
A Lie algebra is a Euclidean $\mathbb{R}$-module endowed with a bilinear
binary operation $\left[  ,\right]  $ (called a Lie bracket) abiding by the
antisymmetric law and the Jacobi identity which is microlinear as a mere space.
\end{definition}

The unit element of a group is usually denoted by $1$, while the unit element
of the underlying abelian group of an $\mathbb{R}$-module is usually denoted
by $0$. Given a Lie group $G$\ and a space $U$, $G^{U}$ is naturally a Lie
group. Similarly, given a Lie algebra $\mathfrak{g}$\ and a space $U$,
$\mathfrak{g}^{U}$ is naturally a Lie algebra.

\begin{notation}
We make it a rule that abbreviated Lie brackets should be inserted from the
right to the left. Therefore
\[
\left[  X,Y,Z\right]
\]
stands for
\[
\left[  X,\left[  Y,Z\right]  \right]
\]
by way of example.
\end{notation}

\section{\label{s3}Weil Prolongations of Lie Groups and their Lie Algebras}

As is the case in the equivalence of the three distinct viewpoints of vector
fields (cf. \S 3.2.1 of \cite{la}), the exponential law plays a significant
role in synthetic differential geometry, which is also the case in the
considerations to follow. Since each Weil algebra has its counterpart in an
adequate model of synthetic differential geometry (called an
\textit{infinitesimal object}), Weil prolongations are merely exponentiations
by infinitesimal objects in synthetic differential geometry. It is not
difficult to externalize Weil prolongations, for which the reader is referred,
e.g., to Chapter VIII of \cite{kms} or \S 31 of \cite{km}. Weil prolongations
play a significant role in axiomatic differential geometry under construction,
for which the reader is referred to \cite{ni1} and \cite{ni2}.

Let us begin by fixing our notation.

\begin{notation}
We denote by
\[
\mathbf{Lie}%
\]
the functor assigning to each Lie group $G$\ its Lie algebra $\mathbf{Lie}%
\left(  G\right)  $\ and to each homomorphism $\varphi:G\rightarrow G^{\prime
}$\ of Lie groups to its induced homomorphism $\mathbf{Lie}\left(
\varphi\right)  :\mathbf{Lie}\left(  G\right)  \rightarrow\mathbf{Lie}\left(
G^{\prime}\right)  $\ of their Lie algebras. We will often write
$\mathfrak{g}$\ for $\mathbf{Lie}\left(  G\right)  $, as is usual.
\end{notation}

\begin{notation}
Given a Lie group $G$, we denote by
\[
\left(  G^{D_{n}}\right)  _{1}%
\]
the subgroup
\[
\left\{  f\in G^{D_{n}}\mid f\left(  0\right)  =1\right\}
\]
of the Lie group $G^{D_{n}}$.
\end{notation}

\begin{notation}
Given a Lie algebra $\mathfrak{g}$, we denote by
\[
\left(  \mathfrak{g}^{D_{n}}\right)  _{0}%
\]
the subalgebra
\[
\left\{  f\in\mathfrak{g}^{D_{n}}\mid f\left(  0\right)  =0\right\}
\]
of the Lie algebra $\mathfrak{g}^{D_{n}}$.
\end{notation}

\begin{theorem}
\label{t3.1}Given a Lie group $G$\ with its Lie algebra $\mathfrak{g}$, we
have
\[
\mathbf{Lie}\left(  G^{D_{n}}\right)  =\mathfrak{g}^{D_{n}}%
\]

\end{theorem}

\begin{proof}
This follows mainly from the familiar exponential law
\[
\left(  G^{D_{n}}\right)  ^{D}=G^{D_{n}\times D}=\left(  G^{D}\right)
^{D_{n}}%
\]
which naturally gives rise, by restriction, to
\begin{align*}
& \left(  \left(  G^{D_{n}}\right)  ^{D}\right)  _{1}\\
& =\left\{  f\in G^{D_{n}\times D}\mid f\left(  d,0\right)  =1\;\left(
\forall d\in D_{n}\right)  \right\} \\
& =\mathfrak{g}^{D_{n}}%
\end{align*}

\end{proof}

\begin{corollary}
\label{t3.1.1}
\[
\mathbf{Lie}\left(  \left(  G^{D_{n}}\right)  _{1}\right)  =\left(
\mathfrak{g}^{D_{n}}\right)  _{0}%
\]

\end{corollary}

\begin{proof}
This follows mainly from the familiar exponential law
\[
\left(  G^{D_{n}}\right)  ^{D}=G^{D_{n}\times D}=\left(  G^{D}\right)
^{D_{n}}%
\]
which naturally gives rise, by restriction, to
\begin{align*}
& \left(  \left(  \left(  G^{D_{n}}\right)  _{1}\right)  ^{D}\right)
_{1^{D_{n}}}\\
& =\left\{  f\in G^{D_{n}\times D}\mid f\left(  d,0\right)  =1\;\left(
\forall d\in D_{n}\right)  \text{ and }f(0,e)=1\;\left(  \forall e\in
D\right)  \right\} \\
& =\left(  \mathfrak{g}^{D_{n}}\right)  _{0}%
\end{align*}

\end{proof}

\begin{corollary}
\label{t3.1.2}Given $\sum_{i=0}^{n}X_{i}d^{i},\sum_{j=0}^{n}Y_{j}d^{j}%
\in\mathfrak{g}^{D_{n}}=\mathbf{Lie}\left(  G^{D_{n}}\right)  $ with $d\in
D_{n}$, we can easily compute their Lie bracket as follows:
\[
\left[  \sum_{i=0}^{n}X_{i}d^{i},\sum_{j=0}^{n}Y_{j}d^{j}\right]  =\sum
_{k=0}^{n}\left(  \sum_{i+j=k}\left[  X_{i},Y_{j}\right]  \right)  d^{k}%
\]

\end{corollary}

\section{\label{s4}Generalized Baker-Campbell-Hausdorff Formulas}

In this section, $G$ is assumed to be a regular Lie group with its Lie algebra
$\mathfrak{g}$. It should be obvious that

\begin{theorem}
\label{t4.1}With $d_{1}\in D$\ and $X_{1},Y_{1}\in\mathfrak{g}$, we have
\begin{align*}
& \exp\,d_{1}X_{1}.\exp\,d_{1}Y_{1}\\
& =\exp\,d_{1}\left(  X_{1}+Y_{1}\right)
\end{align*}

\end{theorem}

\begin{theorem}
\label{t4.2}With $d_{1},d_{2}\in D$\ and $X_{1},X_{2},Y_{1},Y_{2}%
\in\mathfrak{g} $, we have
\begin{align*}
& \exp\,\left(  d_{1}+d_{2}\right)  X_{1}+\frac{1}{2}\left(  d_{1}%
+d_{2}\right)  ^{2}X_{2}.\exp\,\left(  d_{1}+d_{2}\right)  Y_{1}+\frac{1}%
{2}\left(  d_{1}+d_{2}\right)  ^{2}Y_{2}\\
& =\exp\,\left(  d_{1}+d_{2}\right)  \left(  X_{1}+Y_{1}\right)  +\frac{1}%
{2}\left(  d_{1}+d_{2}\right)  ^{2}\left(  X_{2}+Y_{2}+\left[  X_{1}%
,Y_{1}\right]  \right)
\end{align*}

\end{theorem}

\begin{proof}
We have
\begin{align*}
& \exp\,\left(  d_{1}+d_{2}\right)  X_{1}+\frac{1}{2}\left(  d_{1}%
+d_{2}\right)  ^{2}X_{2}.\exp\,\left(  d_{1}+d_{2}\right)  Y_{1}+\frac{1}%
{2}\left(  d_{1}+d_{2}\right)  ^{2}Y_{2}\\
& =\exp\,d_{1}X_{1}+d_{2}\left(  X_{1}+d_{1}X_{2}\right)  .\exp\,d_{1}%
Y_{1}+d_{2}\left(  Y_{1}+d_{1}Y_{2}\right) \\
& =\exp\,d_{2}\left(  X_{1}+d_{1}X_{2}\right)  .\exp\,d_{1}X_{1}.\exp
\,d_{1}Y_{1}.\exp\,d_{2}\left(  Y_{1}+d_{1}Y_{2}\right) \\
& \left)  \text{By Lemmas \ref{t4.2.1} and \ref{t4.2.2}}\right( \\
& =\exp\,d_{2}\left(  X_{1}+d_{1}X_{2}\right)  .\exp\,d_{1}\left(  X_{1}%
+Y_{1}\right)  .\exp\,d_{2}\left(  Y_{1}+d_{1}Y_{2}\right) \\
& \left)  \text{By Theorem \ref{t4.1}}\right( \\
& =\exp\,-\frac{1}{2}d_{1}d_{2}\left[  Y_{1},X_{1}\right]  .\exp\,d_{1}\left(
X_{1}+Y_{1}\right)  +d_{2}\left(  X_{1}+d_{1}X_{2}\right)  .\exp\,d_{2}\left(
Y_{1}+d_{1}Y_{2}\right) \\
& \left)  \text{By Lemma \ref{t4.2.3}}\right( \\
& =\exp\,-\frac{1}{2}d_{1}d_{2}\left[  Y_{1},X_{1}\right]  .\exp\,d_{1}\left(
X_{1}+Y_{1}\right)  +d_{2}\left(  X_{1}+d_{1}X_{2}\right)  +d_{2}\left(
Y_{1}+d_{1}Y_{2}\right)  .\\
& \exp\,\frac{1}{2}d_{1}d_{2}\left[  X_{1},Y_{1}\right] \\
& \left)  \text{By Lemma \ref{t4.2.4}}\right(
\end{align*}
so that we have the desired formula.
\end{proof}

\begin{lemma}
\label{t4.2.1}
\begin{align*}
& \exp\,d_{1}X_{1}+d_{2}\left(  X_{1}+d_{1}X_{2}\right) \\
& =\exp\,d_{2}\left(  X_{1}+d_{1}X_{2}\right)  .\exp\,d_{1}X_{1}%
\end{align*}

\end{lemma}

\begin{proof}
Letting $\ast_{1}$ denote
\[
d_{1}X_{1}%
\]
and letting $\ast_{2}$ denote
\[
X_{1}+d_{1}X_{2}\text{,}%
\]
we have
\begin{align}
& \exp\,\ast_{1}+d_{2}\ast_{2}\nonumber\\
& =\exp\,d_{2}\ast_{2}.\exp\,\ast_{1}\label{4.2.1.1}%
\end{align}
by right logarithmic derivation. Therefore the desired formula follows.
\end{proof}

\begin{lemma}
\label{t4.2.2}
\begin{align*}
& \exp\,d_{1}Y_{1}+d_{2}\left(  Y_{1}+d_{1}Y_{2}\right) \\
& =\exp\,d_{1}Y_{1}.\exp\,d_{2}\left(  Y_{1}+d_{1}Y_{2}\right)
\end{align*}

\end{lemma}

\begin{proof}
Letting $\ast_{1}$ denote
\[
d_{1}Y_{1}%
\]
and letting $\ast_{2}$ denote
\[
Y_{1}+d_{1}Y_{2}\text{,}%
\]
we have
\begin{align}
& \exp\,\ast_{1}+d_{2}\ast_{2}\nonumber\\
& =\exp\,\ast_{1}.\exp\,d_{2}\ast_{2}\label{4.2.2.1}%
\end{align}
by left logarithmic derivation. Therefore the desired formula follows.
\end{proof}

\begin{lemma}
\label{t4.2.3}We have
\begin{align*}
& \exp\,d_{1}\left(  X_{1}+Y_{1}\right)  +d_{2}\left(  X_{1}+d_{1}X_{2}\right)
\\
& =\exp\,d_{2}\left(  X_{1}+d_{1}X_{2}+\frac{1}{2}d_{1}\left[  Y_{1}%
,X_{1}\right]  \right)  .\exp\,d_{1}\left(  X_{1}+Y_{1}\right)
\end{align*}

\end{lemma}

\begin{proof}
Letting $\ast_{1}$ denote
\[
d_{1}\left(  X_{1}+Y_{1}\right)
\]
and letting $\ast_{2}$ denote
\[
X_{1}+d_{1}X_{2}\text{,}%
\]
we have
\begin{align}
& \exp\,\ast_{1}+d_{2}\ast_{2}\nonumber\\
& =\exp\,d_{2}\left\{  \ast_{2}+\frac{1}{2}\left[  \ast_{1},\ast_{2}\right]
\right\}  .\exp\,\ast_{1}\label{4.2.3.1}%
\end{align}
by right logarithmic derivation. By the way, we have the following:
\[
\left[  \ast_{1},\ast_{2}\right]  =d_{1}\left[  Y_{1},X_{1}\right]
\]
Therefore the desired formula follows.
\end{proof}

\begin{lemma}
\label{t4.2.4}We have
\begin{align*}
& \exp\,d_{1}\left(  X_{1}+Y_{1}\right)  +d_{2}\left(  X_{1}+d_{1}%
X_{2}\right)  +d_{2}\left(  Y_{1}+d_{1}Y_{2}\right) \\
& =\exp\,d_{1}\left(  X_{1}+Y_{1}\right)  +d_{2}\left(  X_{1}+d_{1}%
X_{2}\right)  .\exp\,d_{2}\left(  Y_{1}+d_{1}Y_{2}-\frac{1}{2}d_{1}\left[
X_{1},Y_{1}\right]  \right)
\end{align*}

\end{lemma}

\begin{proof}
Letting $\ast_{1}$ denote
\[
d_{1}\left(  X_{1}+Y_{1}\right)  +d_{2}\left(  X_{1}+d_{1}X_{2}\right)
\]
and letting $\ast_{2}$ denote
\[
Y_{1}+d_{1}Y_{2}\text{,}%
\]
we have
\begin{align}
& \exp\,\ast_{1}+d_{2}\ast_{2}\nonumber\\
& =\exp\,\ast_{1}.\exp\,d_{2}\left\{  \ast_{2}-\frac{1}{2}\left[  \ast
_{1},\ast_{2}\right]  \right\}  .\label{4.2.4.1}%
\end{align}
by left logarithmic derivation. By the way, we have the following:
\[
\left[  \ast_{1},\ast_{2}\right]  =d_{1}\left[  X_{1},Y_{1}\right]
\]
Therefore the desired formula follows.
\end{proof}

\begin{theorem}
\label{t4.3}With $d_{1},d_{2},d_{3}\in D$\ and $X_{1},X_{2},X_{3},Y_{1}%
,Y_{2},Y_{3}\in\mathfrak{g}$, we have
\begin{align*}
& \exp\,\left(  d_{1}+d_{2}+d_{3}\right)  X_{1}+\frac{1}{2}\left(  d_{1}%
+d_{2}+d_{3}\right)  ^{2}X_{2}+\frac{1}{6}\left(  d_{1}+d_{2}+d_{3}\right)
^{3}X_{3}.\\
& \exp\,\left(  d_{1}+d_{2}+d_{3}\right)  Y_{1}+\frac{1}{2}\left(  d_{1}%
+d_{2}+d_{3}\right)  ^{2}Y_{2}+\frac{1}{6}\left(  d_{1}+d_{2}+d_{3}\right)
^{3}Y_{3}\\
& =\exp\,\left(  d_{1}+d_{2}+d_{3}\right)  \left(  X_{1}+Y_{1}\right)
+\frac{1}{2}\left(  d_{1}+d_{2}+d_{3}\right)  ^{2}\left(  X_{2}+Y_{2}+\left[
X_{1},Y_{1}\right]  \right)  +\\
& \frac{1}{6}\left(  d_{1}+d_{2}+d_{3}\right)  ^{3}\left\{  \left(
X_{3}+Y_{3}\right)  +\frac{3}{2}\left(  \left[  X_{1},Y_{2}\right]  +\left[
X_{2},Y_{1}\right]  \right)  +\frac{1}{2}\left[  X_{1}-Y_{1},X_{1}%
,Y_{1}\right]  \right\}
\end{align*}

\end{theorem}

\begin{proof}
We have
\begin{align*}
& \exp\,\left(  d_{1}+d_{2}+d_{3}\right)  X_{1}+\frac{1}{2}\left(  d_{1}%
+d_{2}+d_{3}\right)  ^{2}X_{2}+\frac{1}{6}\left(  d_{1}+d_{2}+d_{3}\right)
^{3}X_{3}.\\
& \exp\,\left(  d_{1}+d_{2}+d_{3}\right)  Y_{1}+\frac{1}{2}\left(  d_{1}%
+d_{2}+d_{3}\right)  ^{2}Y_{2}+\frac{1}{6}\left(  d_{1}+d_{2}+d_{3}\right)
^{3}Y_{3}\\
& =\exp\,\left\{  \left(  d_{1}+d_{2}\right)  X_{1}+\frac{1}{2}\left(
d_{1}+d_{2}\right)  ^{2}X_{2}\right\}  +d_{3}\left\{  X_{1}+\left(
d_{1}+d_{2}\right)  X_{2}+\frac{1}{2}\left(  d_{1}+d_{2}\right)  ^{2}%
X_{3}\right\}  .\\
& \exp\,\left\{  \left(  d_{1}+d_{2}\right)  Y_{1}+\frac{1}{2}\left(
d_{1}+d_{2}\right)  ^{2}Y_{2}\right\}  +d_{3}\left\{  Y_{1}+\left(
d_{1}+d_{2}\right)  Y_{2}+\frac{1}{2}\left(  d_{1}+d_{2}\right)  ^{2}%
Y_{3}\right\} \\
& =\exp\,d_{3}\left\{  X_{1}+\left(  d_{1}+d_{2}\right)  X_{2}+\left(
d_{1}+d_{2}\right)  ^{2}\left(  \frac{1}{2}X_{3}+\frac{1}{4}\left[
X_{1},X_{2}\right]  \right)  \right\}  .\\
& \exp\,\left(  d_{1}+d_{2}\right)  X_{1}+\frac{1}{2}\left(  d_{1}%
+d_{2}\right)  ^{2}X_{2}.\exp\,\left(  d_{1}+d_{2}\right)  Y_{1}+\frac{1}%
{2}\left(  d_{1}+d_{2}\right)  ^{2}Y_{2}.\\
& \exp\,d_{3}\left\{  Y_{1}+\left(  d_{1}+d_{2}\right)  Y_{2}+\frac{1}%
{2}\left(  d_{1}+d_{2}\right)  ^{2}\left(  Y_{3}-\frac{1}{2}\left[
Y_{1},Y_{2}\right]  \right)  \right\} \\
& \left)  \text{By Lemmas \ref{t4.3.1}\ and \ref{t4.3.2}}\right( \\
& =\exp\,d_{3}\left\{  X_{1}+\left(  d_{1}+d_{2}\right)  X_{2}+\frac{1}%
{2}\left(  d_{1}+d_{2}\right)  ^{2}X_{3}+\frac{1}{4}\left(  d_{1}%
+d_{2}\right)  ^{2}\left[  X_{1},X_{2}\right]  \right\}  .\\
& \exp\,\left(  d_{1}+d_{2}\right)  \left(  X_{1}+Y_{1}\right)  +\frac{1}%
{2}\left(  d_{1}+d_{2}\right)  ^{2}\left(  X_{2}+Y_{2}+\left[  X_{1}%
,Y_{1}\right]  \right)  .\\
& \exp\,d_{3}\left\{  Y_{1}+\left(  d_{1}+d_{2}\right)  Y_{2}+\frac{1}%
{2}\left(  d_{1}+d_{2}\right)  ^{2}Y_{3}-\frac{1}{4}\left(  d_{1}%
+d_{2}\right)  ^{2}\left[  Y_{1},Y_{2}\right]  \right\} \\
& \left)  \text{By Theorem \ref{t4.2}}\right( \\
& =\exp\,-d_{3}\left\{
\begin{array}
[c]{c}%
\frac{1}{2}\left(  d_{1}+d_{2}\right)  \left[  Y_{1},X_{1}\right]  +\\
\frac{1}{2}\left(  d_{1}+d_{2}\right)  ^{2}\left(
\begin{array}
[c]{c}%
\left[  X_{1}+Y_{1},X_{2}\right]  +\\
\frac{1}{2}\left[  X_{2}+Y_{2}+\left[  X_{1},Y_{1}\right]  ,X_{1}\right]  +\\
\frac{1}{3}\left[  X_{1}+Y_{1},Y_{1},X_{1}\right]
\end{array}
\right)
\end{array}
\right\}  .\\
& \exp\,\left(  d_{1}+d_{2}\right)  \left(  X_{1}+Y_{1}\right)  +\frac{1}%
{2}\left(  d_{1}+d_{2}\right)  ^{2}\left(  X_{2}+Y_{2}+\left[  X_{1}%
,Y_{1}\right]  \right)  +\\
& d_{3}\left\{  X_{1}+\left(  d_{1}+d_{2}\right)  X_{2}+\frac{1}{2}\left(
d_{1}+d_{2}\right)  ^{2}\left(  X_{3}+\frac{1}{2}\left[  X_{1},X_{2}\right]
\right)  \right\}  .\\
& \exp\,d_{3}\left\{  Y_{1}+\left(  d_{1}+d_{2}\right)  Y_{2}+\frac{1}%
{2}\left(  d_{1}+d_{2}\right)  ^{2}\left(  Y_{3}-\frac{1}{2}\left[
Y_{1},Y_{2}\right]  \right)  \right\} \\
& \left)  \text{By Lemma \ref{t4.3.3}}\right(
\end{align*}
We keep on.
\begin{align*}
& =\exp\,\frac{1}{4}\left(  d_{1}+d_{2}\right)  ^{2}d_{3}\left[  X_{1}%
+Y_{1},Y_{1},X_{1}\right]  .\\
& \exp\,\left(  d_{1}+d_{2}\right)  \left(  X_{1}+Y_{1}\right)  +\frac{1}%
{2}\left(  d_{1}+d_{2}\right)  ^{2}\left(  X_{2}+Y_{2}+\left[  X_{1}%
,Y_{1}\right]  \right)  +\\
& d_{3}\left\{  X_{1}+\left(  d_{1}+d_{2}\right)  X_{2}+\frac{1}{2}\left(
d_{1}+d_{2}\right)  ^{2}\left(  X_{3}+\frac{1}{2}\left[  X_{1},X_{2}\right]
\right)  \right\}  -\\
& d_{3}\left\{
\begin{array}
[c]{c}%
\frac{1}{2}\left(  d_{1}+d_{2}\right)  \left[  Y_{1},X_{1}\right]  +\\
\frac{1}{2}\left(  d_{1}+d_{2}\right)  ^{2}\left(
\begin{array}
[c]{c}%
\left[  X_{1}+Y_{1},X_{2}\right]  +\frac{1}{2}\left[  X_{2}+Y_{2}+\left[
X_{1},Y_{1}\right]  ,X_{1}\right]  +\\
\frac{1}{3}\left[  X_{1}+Y_{1},Y_{1},X_{1}\right]
\end{array}
\right)
\end{array}
\right\}  .\\
& \exp\,d_{3}\left\{  Y_{1}+\left(  d_{1}+d_{2}\right)  Y_{2}+\frac{1}%
{2}\left(  d_{1}+d_{2}\right)  ^{2}\left(  Y_{3}-\frac{1}{2}\left[
Y_{1},Y_{2}\right]  \right)  \right\} \\
& \left)  \text{By Lemma \ref{t4.3.4}}\right( \\
& =\exp\,\frac{1}{4}\left(  d_{1}+d_{2}\right)  ^{2}d_{3}\left[  X_{1}%
+Y_{1},Y_{1},X_{1}\right]  .\\
& \exp\,\left(  d_{1}+d_{2}\right)  \left(  X_{1}+Y_{1}\right)  +\frac{1}%
{2}\left(  d_{1}+d_{2}\right)  ^{2}\left(  X_{2}+Y_{2}+\left[  X_{1}%
,Y_{1}\right]  \right)  +\\
& d_{3}\left\{  X_{1}+\left(  d_{1}+d_{2}\right)  X_{2}+\frac{1}{2}\left(
d_{1}+d_{2}\right)  ^{2}\left(  X_{3}+\frac{1}{2}\left[  X_{1},X_{2}\right]
\right)  \right\}  -\\
& d_{3}\left\{
\begin{array}
[c]{c}%
\frac{1}{2}\left(  d_{1}+d_{2}\right)  \left[  Y_{1},X_{1}\right]  +\\
\frac{1}{2}\left(  d_{1}+d_{2}\right)  ^{2}\left(
\begin{array}
[c]{c}%
\left[  X_{1}+Y_{1},X_{2}\right]  +\frac{1}{2}\left[  X_{2}+Y_{2}+\left[
X_{1},Y_{1}\right]  ,X_{1}\right]  +\\
\frac{1}{3}\left[  X_{1}+Y_{1},Y_{1},X_{1}\right]
\end{array}
\right)
\end{array}
\right\}  +\\
& d_{3}\left\{  Y_{1}+\left(  d_{1}+d_{2}\right)  Y_{2}+\frac{1}{2}\left(
d_{1}+d_{2}\right)  ^{2}\left(  Y_{3}-\frac{1}{2}\left[  Y_{1},Y_{2}\right]
\right)  \right\}  .\\
& \exp\,d_{3}\left\{
\begin{array}
[c]{c}%
\frac{1}{2}\left(  d_{1}+d_{2}\right)  \left(  d_{1}+d_{2}\right)  \left[
X_{1},Y_{1}\right] \\
\frac{1}{2}\left(  d_{1}+d_{2}\right)  ^{2}\left(
\begin{array}
[c]{c}%
\left[  X_{1}+Y_{1},Y_{2}\right]  +\frac{1}{2}\left[  X_{2}+Y_{2}+\left[
X_{1},Y_{1}\right]  ,Y_{1}\right]  -\\
\frac{1}{3}\left[  X_{1}+Y_{1},Y_{1},X_{1}\right]
\end{array}
\right)
\end{array}
\right\} \\
& \left)  \text{By Lemma \ref{t4.3.5}}\right(
\end{align*}
We keep on again.
\begin{align*}
& =\exp\,\frac{1}{4}\left(  d_{1}+d_{2}\right)  ^{2}d_{3}\left[  X_{1}%
+Y_{1},Y_{1},X_{1}\right]  .\\
& \exp\,\left(  d_{1}+d_{2}\right)  \left(  X_{1}+Y_{1}\right)  +\frac{1}%
{2}\left(  d_{1}+d_{2}\right)  ^{2}\left(  X_{2}+Y_{2}+\left[  X_{1}%
,Y_{1}\right]  \right)  +\\
& d_{3}\left\{  X_{1}+\left(  d_{1}+d_{2}\right)  X_{2}+\frac{1}{2}\left(
d_{1}+d_{2}\right)  ^{2}\left(  X_{3}+\frac{1}{2}\left[  X_{1},X_{2}\right]
\right)  \right\}  -\\
& d_{3}\left\{
\begin{array}
[c]{c}%
\frac{1}{2}\left(  d_{1}+d_{2}\right)  \left[  Y_{1},X_{1}\right]  +\\
\frac{1}{2}\left(  d_{1}+d_{2}\right)  ^{2}\left(
\begin{array}
[c]{c}%
\left[  X_{1}+Y_{1},X_{2}\right]  +\frac{1}{2}\left[  X_{2}+Y_{2}+\left[
X_{1},Y_{1}\right]  ,X_{1}\right]  +\\
\frac{1}{3}\left[  X_{1}+Y_{1},Y_{1},X_{1}\right]
\end{array}
\right)
\end{array}
\right\}  +\\
& d_{3}\left\{  Y_{1}+\left(  d_{1}+d_{2}\right)  Y_{2}+\frac{1}{2}\left(
d_{1}+d_{2}\right)  ^{2}\left(  Y_{3}-\frac{1}{2}\left[  Y_{1},Y_{2}\right]
\right)  \right\}  +\\
& d_{3}\left\{
\begin{array}
[c]{c}%
\frac{1}{2}\left(  d_{1}+d_{2}\right)  \left[  X_{1},Y_{1}\right]  +\\
\frac{1}{2}\left(  d_{1}+d_{2}\right)  ^{2}\left(
\begin{array}
[c]{c}%
\left[  X_{1}+Y_{1},Y_{2}\right]  +\frac{1}{2}\left[  X_{2}+Y_{2}+\left[
X_{1},Y_{1}\right]  ,Y_{1}\right]  -\\
\frac{1}{3}\left[  X_{1}+Y_{1},Y_{1},X_{1}\right]
\end{array}
\right)
\end{array}
\right\}  .\\
& \exp\,\frac{1}{4}\left(  d_{1}+d_{2}\right)  ^{2}d_{3}\left[  X_{1}%
+Y_{1},X_{1},Y_{1}\right]
\end{align*}
Therefore the desired formula follows at once.
\end{proof}

\begin{lemma}
\label{t4.3.1}
\begin{align*}
& \exp\,\left\{  \left(  d_{1}+d_{2}\right)  X_{1}+\frac{1}{2}\left(
d_{1}+d_{2}\right)  ^{2}X_{2}\right\}  +d_{3}\left\{  X_{1}+\left(
d_{1}+d_{2}\right)  X_{2}+\frac{1}{2}\left(  d_{1}+d_{2}\right)  ^{2}%
X_{3}\right\} \\
& =\exp\,d_{3}\left\{  X_{1}+\left(  d_{1}+d_{2}\right)  X_{2}+\frac{1}%
{2}\left(  d_{1}+d_{2}\right)  ^{2}\left(  X_{3}+\frac{1}{2}\left[
X_{1},X_{2}\right]  \right)  \right\}  .\\
& \exp\,\left\{  \left(  d_{1}+d_{2}\right)  X_{1}+\frac{1}{2}\left(
d_{1}+d_{2}\right)  ^{2}X_{2}\right\}
\end{align*}

\end{lemma}

\begin{proof}
Letting $\ast_{1}$ denote
\[
\left(  d_{1}+d_{2}\right)  X_{1}+\frac{1}{2}\left(  d_{1}+d_{2}\right)
^{2}X_{2}%
\]
and letting $\ast_{2}$ denote
\[
X_{1}+\left(  d_{1}+d_{2}\right)  X_{2}+\frac{1}{2}\left(  d_{1}+d_{2}\right)
^{2}X_{3}\text{,}%
\]
we have
\begin{align}
& \exp\,\ast_{1}+d_{3}\ast_{2}\nonumber\\
& =\exp\,d_{3}\left(  \ast_{2}+\frac{1}{2}\left[  \ast_{1},\ast_{2}\right]
\right)  .\exp\,\ast_{1}\label{4.3.1.1}%
\end{align}
by right logarithmic derivation. By the way, we have the following:
\begin{align*}
\left[  \ast_{1},\ast_{2}\right]   & =\left(  d_{1}+d_{2}\right)  ^{2}\left(
\left[  X_{1},X_{2}\right]  +\frac{1}{2}\left[  X_{2},X_{1}\right]  \right) \\
& =\frac{1}{2}\left(  d_{1}+d_{2}\right)  ^{2}\left[  X_{1},X_{2}\right]
\end{align*}
Therefore the desired formula follows.
\end{proof}

\begin{lemma}
\label{t4.3.2}
\begin{align*}
& \exp\,\left\{  \left(  d_{1}+d_{2}\right)  Y_{1}+\frac{1}{2}\left(
d_{1}+d_{2}\right)  ^{2}Y_{2}\right\}  +d_{3}\left\{  Y_{1}+\left(
d_{1}+d_{2}\right)  Y_{2}+\frac{1}{2}\left(  d_{1}+d_{2}\right)  ^{2}%
Y_{3}\right\} \\
& =\exp\,\left(  d_{1}+d_{2}\right)  Y_{1}+\frac{1}{2}\left(  d_{1}%
+d_{2}\right)  ^{2}Y_{2}.\\
& \exp\,d_{3}\left\{  Y_{1}+\left(  d_{1}+d_{2}\right)  Y_{2}+\frac{1}%
{2}\left(  d_{1}+d_{2}\right)  ^{2}\left(  Y_{3}-\frac{1}{2}\left[
Y_{1},Y_{2}\right]  \right)  \right\}
\end{align*}

\end{lemma}

\begin{proof}
Letting $\ast_{1}$ denote
\[
\left(  d_{1}+d_{2}\right)  Y_{1}+\frac{1}{2}\left(  d_{1}+d_{2}\right)
^{2}Y_{2}%
\]
and letting $\ast_{2}$ denote
\[
Y_{1}+\left(  d_{1}+d_{2}\right)  Y_{2}+\frac{1}{2}\left(  d_{1}+d_{2}\right)
^{2}Y_{3}\text{,}%
\]
we have
\begin{align}
& \exp\,\ast_{1}+d_{3}\ast_{2}\nonumber\\
& =\exp\,\ast_{1}.\exp\,d_{3}\left(  \ast_{2}-\frac{1}{2}\left[  \ast_{1}%
,\ast_{2}\right]  \right) \label{4.3.2.1}%
\end{align}
by left logarithmic derivation. By the way, we have the following:
\begin{align*}
\left[  \ast_{1},\ast_{2}\right]   & =\left(  d_{1}+d_{2}\right)  ^{2}\left(
\left[  Y_{1},Y_{2}\right]  +\frac{1}{2}\left[  Y_{2},Y_{1}\right]  \right) \\
& =\frac{1}{2}\left(  d_{1}+d_{2}\right)  ^{2}\left[  Y_{1},Y_{2}\right]
\end{align*}
Therefore the desired formula follows.
\end{proof}

\begin{lemma}
\label{t4.3.3}
\begin{align*}
& \exp\,\left\{  \left(  d_{1}+d_{2}\right)  \left(  X_{1}+Y_{1}\right)
+\frac{1}{2}\left(  d_{1}+d_{2}\right)  ^{2}\left(  X_{2}+Y_{2}+\left[
X_{1},Y_{1}\right]  \right)  \right\}  +\\
& d_{3}\left\{  X_{1}+\left(  d_{1}+d_{2}\right)  X_{2}+\frac{1}{2}\left(
d_{1}+d_{2}\right)  ^{2}\left(  X_{3}+\frac{1}{2}\left[  X_{1},X_{2}\right]
\right)  \right\} \\
& =\exp\,d_{3}\left\{
\begin{array}
[c]{c}%
X_{1}+\left(  d_{1}+d_{2}\right)  X_{2}+\frac{1}{2}\left(  d_{1}+d_{2}\right)
^{2}\left(  X_{3}+\frac{1}{2}\left[  X_{1},X_{2}\right]  \right)  +\\
\frac{1}{2}\left(  d_{1}+d_{2}\right)  \left[  Y_{1},X_{1}\right]  +\\
\left(  d_{1}+d_{2}\right)  ^{2}\left(
\begin{array}
[c]{c}%
\frac{1}{2}\left[  X_{1}+Y_{1},X_{2}\right]  +\frac{1}{4}\left[  X_{2}%
+Y_{2}+\left[  X_{1},Y_{1}\right]  ,X_{1}\right]  +\\
\frac{1}{6}\left[  X_{1}+Y_{1},Y_{1},X_{1}\right]
\end{array}
\right)
\end{array}
\right\}  .\\
& \exp\,\left(  d_{1}+d_{2}\right)  \left(  X_{1}+Y_{1}\right)  +\frac{1}%
{2}\left(  d_{1}+d_{2}\right)  ^{2}\left(  X_{2}+Y_{2}+\left[  X_{1}%
,Y_{1}\right]  \right)
\end{align*}

\end{lemma}

\begin{proof}
Letting $\ast_{1}$ denote
\[
\left(  d_{1}+d_{2}\right)  \left(  X_{1}+Y_{1}\right)  +\frac{1}{2}\left(
d_{1}+d_{2}\right)  ^{2}\left(  X_{2}+Y_{2}+\left[  X_{1},Y_{1}\right]
\right)
\]
and letting $\ast_{2}$ denote
\[
X_{1}+\left(  d_{1}+d_{2}\right)  X_{2}+\frac{1}{2}\left(  d_{1}+d_{2}\right)
^{2}\left(  X_{3}+\frac{1}{2}\left[  X_{1},X_{2}\right]  \right)  \text{,}%
\]
we have
\begin{align}
& \exp\,\ast_{1}+d_{3}\ast_{2}\nonumber\\
& =\exp\,d_{3}\left(  \ast_{2}+\frac{1}{2}\left[  \ast_{1},\ast_{2}\right]
+\frac{1}{6}\left[  \ast_{1},\ast_{1},\ast_{2}\right]  \right)  .\exp
\,\ast_{1}\label{4.3.3.1}%
\end{align}
by right logarithmic derivation. By the way, we have the following:
\[
\left[  \ast_{1},\ast_{2}\right]  =\left(  d_{1}+d_{2}\right)  \left[
Y_{1},X_{1}\right]  +\left(  d_{1}+d_{2}\right)  ^{2}\left(  \left[
X_{1}+Y_{1},X_{2}\right]  +\frac{1}{2}\left[  X_{2}+Y_{2}+\left[  X_{1}%
,Y_{1}\right]  ,X_{1}\right]  \right)
\]
\[
\left[  \ast_{1},\ast_{1},\ast_{2}\right]  =\left(  d_{1}+d_{2}\right)
^{2}\left[  X_{1}+Y_{1},Y_{1},X_{1}\right]
\]
\begin{align*}
& \frac{1}{2}\left[  \ast_{1},\ast_{2}\right]  +\frac{1}{6}\left[  \ast
_{1},\ast_{1},\ast_{2}\right] \\
& =\frac{1}{2}\left(  d_{1}+d_{2}\right)  \left[  Y_{1},X_{1}\right]  +\\
& \left(  d_{1}+d_{2}\right)  ^{2}\left(  \frac{1}{2}\left[  X_{1}+Y_{1}%
,X_{2}\right]  +\frac{1}{4}\left[  X_{2}+Y_{2}+\left[  X_{1},Y_{1}\right]
,X_{1}\right]  +\frac{1}{6}\left[  X_{1}+Y_{1},Y_{1},X_{1}\right]  \right)
\end{align*}
Therefore the desired formula follows.
\end{proof}

\begin{lemma}
\label{t4.3.4}
\begin{align*}
& \exp\,\left(  d_{1}+d_{2}\right)  \left(  X_{1}+Y_{1}\right)  +\frac{1}%
{2}\left(  d_{1}+d_{2}\right)  ^{2}\left(  X_{2}+Y_{2}+\left[  X_{1}%
,Y_{1}\right]  \right)  +\\
& d_{3}\left\{  X_{1}+\left(  d_{1}+d_{2}\right)  X_{2}+\frac{1}{2}\left(
d_{1}+d_{2}\right)  ^{2}\left(  X_{3}+\frac{1}{2}\left[  X_{1},X_{2}\right]
\right)  \right\}  -\\
& d_{3}\left\{
\begin{array}
[c]{c}%
\frac{1}{2}\left(  d_{1}+d_{2}\right)  \left[  Y_{1},X_{1}\right]  +\\
\frac{1}{2}\left(  d_{1}+d_{2}\right)  ^{2}\left(
\begin{array}
[c]{c}%
\left[  X_{1}+Y_{1},X_{2}\right]  +\frac{1}{2}\left[  X_{2}+Y_{2}+\left[
X_{1},Y_{1}\right]  ,X_{1}\right]  +\\
\frac{1}{3}\left[  X_{1}+Y_{1},Y_{1},X_{1}\right]
\end{array}
\right)
\end{array}
\right\} \\
& =\exp\,-d_{3}\left\{
\begin{array}
[c]{c}%
\frac{1}{2}\left(  d_{1}+d_{2}\right)  \left[  Y_{1},X_{1}\right]  +\\
\left(  d_{1}+d_{2}\right)  ^{2}\left(
\begin{array}
[c]{c}%
\frac{1}{2}\left[  X_{1}+Y_{1},X_{2}\right]  +\frac{1}{4}\left[  X_{2}%
+Y_{2}+\left[  X_{1},Y_{1}\right]  ,X_{1}\right]  +\\
\frac{1}{6}\left[  X_{1}+Y_{1},Y_{1},X_{1}\right]
\end{array}
\right)  +\\
\frac{1}{4}\left(  d_{1}+d_{2}\right)  ^{2}\left[  X_{1}+Y_{1},Y_{1}%
,X_{1}\right]
\end{array}
\right\}  .\\
& \exp\,\left(  d_{1}+d_{2}\right)  \left(  X_{1}+Y_{1}\right)  +\frac{1}%
{2}\left(  d_{1}+d_{2}\right)  ^{2}\left(  X_{2}+Y_{2}+\left[  X_{1}%
,Y_{1}\right]  \right)  +\\
& d_{3}\left\{  X_{1}+\left(  d_{1}+d_{2}\right)  X_{2}+\frac{1}{2}\left(
d_{1}+d_{2}\right)  ^{2}\left(  X_{3}+\frac{1}{2}\left[  X_{1},X_{2}\right]
\right)  \right\}  .
\end{align*}

\end{lemma}

\begin{proof}
Letting $\ast_{1}$ denote
\begin{align*}
& \,\left(  d_{1}+d_{2}\right)  \left(  X_{1}+Y_{1}\right)  +\frac{1}%
{2}\left(  d_{1}+d_{2}\right)  ^{2}\left(  X_{2}+Y_{2}+\left[  X_{1}%
,Y_{1}\right]  \right)  +\\
& d_{3}\left\{  X_{1}+\left(  d_{1}+d_{2}\right)  X_{2}+\frac{1}{2}\left(
d_{1}+d_{2}\right)  ^{2}\left(  X_{3}+\frac{1}{2}\left[  X_{1},X_{2}\right]
\right)  \right\}  -\\
& d_{3}\left\{
\begin{array}
[c]{c}%
\frac{1}{2}\left(  d_{1}+d_{2}\right)  \left[  Y_{1},X_{1}\right]  +\\
\frac{1}{2}\left(  d_{1}+d_{2}\right)  ^{2}\left(
\begin{array}
[c]{c}%
\left[  X_{1}+Y_{1},X_{2}\right]  +\frac{1}{2}\left[  X_{2}+Y_{2}+\left[
X_{1},Y_{1}\right]  ,X_{1}\right]  +\\
\frac{1}{3}\left[  X_{1}+Y_{1},Y_{1},X_{1}\right]
\end{array}
\right)
\end{array}
\right\}
\end{align*}
and letting $\ast_{2}$ denote
\begin{align*}
& -\frac{1}{2}\left(  d_{1}+d_{2}\right)  \left[  Y_{1},X_{1}\right]  -\\
& \frac{1}{2}\left(  d_{1}+d_{2}\right)  ^{2}\left(
\begin{array}
[c]{c}%
\left[  X_{1}+Y_{1},X_{2}\right]  +\frac{1}{2}\left[  X_{2}+Y_{2}+\left[
X_{1},Y_{1}\right]  ,X_{1}\right]  +\\
\frac{1}{3}\left[  X_{1}+Y_{1},Y_{1},X_{1}\right]
\end{array}
\right)
\end{align*}
we have (\ref{4.3.1.1}). By the way, we have the following:
\[
\left[  \ast_{1},\ast_{2}\right]  =-\frac{1}{2}\left(  d_{1}+d_{2}\right)
^{2}\left[  X_{1}+Y_{1},Y_{1},X_{1}\right]
\]
Therefore the desired formula follows at once.
\end{proof}

\begin{lemma}
\label{t4.3.5}
\begin{align*}
& \exp\,\left\{  \left(  d_{1}+d_{2}\right)  \left(  X_{1}+Y_{1}\right)
+\frac{1}{2}\left(  d_{1}+d_{2}\right)  ^{2}\left(  X_{2}+Y_{2}+\left[
X_{1},Y_{1}\right]  \right)  \right\}  +\\
& d_{3}\left\{  X_{1}+\left(  d_{1}+d_{2}\right)  X_{2}+\frac{1}{2}\left(
d_{1}+d_{2}\right)  ^{2}\left(  X_{3}+\frac{1}{2}\left[  X_{1},X_{2}\right]
\right)  \right\}  -\\
& d_{3}\left\{
\begin{array}
[c]{c}%
\frac{1}{2}\left(  d_{1}+d_{2}\right)  \left[  Y_{1},X_{1}\right]  +\\
\frac{1}{2}\left(  d_{1}+d_{2}\right)  ^{2}\left(
\begin{array}
[c]{c}%
\left[  X_{1}+Y_{1},X_{2}\right]  +\frac{1}{2}\left[  X_{2}+Y_{2}+\left[
X_{1},Y_{1}\right]  ,X_{1}\right]  +\\
\frac{1}{3}\left[  X_{1}+Y_{1},Y_{1},X_{1}\right]
\end{array}
\right)
\end{array}
\right\}  +\\
& d_{3}\left\{  Y_{1}+\left(  d_{1}+d_{2}\right)  Y_{2}+\frac{1}{2}\left(
d_{1}+d_{2}\right)  ^{2}\left(  Y_{3}-\frac{1}{2}\left[  Y_{1},Y_{2}\right]
\right)  \right\} \\
& =\exp\,\left\{  \left(  d_{1}+d_{2}\right)  \left(  X_{1}+Y_{1}\right)
+\frac{1}{2}\left(  d_{1}+d_{2}\right)  ^{2}\left(  X_{2}+Y_{2}+\left[
X_{1},Y_{1}\right]  \right)  \right\}  +\\
& d_{3}\left\{  X_{1}+\left(  d_{1}+d_{2}\right)  X_{2}+\frac{1}{2}\left(
d_{1}+d_{2}\right)  ^{2}\left(  X_{3}+\frac{1}{2}\left[  X_{1},X_{2}\right]
\right)  \right\}  -\\
& d_{3}\left\{
\begin{array}
[c]{c}%
\frac{1}{2}\left(  d_{1}+d_{2}\right)  \left[  X_{1},Y_{1}\right] \\
\frac{1}{2}\left(  d_{1}+d_{2}\right)  ^{2}\left(
\begin{array}
[c]{c}%
\left[  X_{1}+Y_{1},Y_{2}\right]  +\frac{1}{2}\left[  X_{2}+Y_{2}+\left[
X_{1},Y_{1}\right]  ,Y_{1}\right]  -\\
\frac{1}{3}\left[  X_{1}+Y_{1},Y_{1},X_{1}\right]
\end{array}
\right)
\end{array}
\right\}
\end{align*}

\end{lemma}

\begin{proof}
Letting $\ast_{1}$ denote
\begin{align*}
& \,\left(  d_{1}+d_{2}\right)  \left(  X_{1}+Y_{1}\right)  +\frac{1}%
{2}\left(  d_{1}+d_{2}\right)  ^{2}\left(  X_{2}+Y_{2}+\left[  X_{1}%
,Y_{1}\right]  \right)  +\\
& d_{3}\left\{  X_{1}+\left(  d_{1}+d_{2}\right)  X_{2}+\frac{1}{2}\left(
d_{1}+d_{2}\right)  ^{2}\left(  X_{3}+\frac{1}{2}\left[  X_{1},X_{2}\right]
\right)  \right\}  -\\
& d_{3}\left\{
\begin{array}
[c]{c}%
\frac{1}{2}\left(  d_{1}+d_{2}\right)  \left[  Y_{1},X_{1}\right]  +\\
\frac{1}{2}\left(  d_{1}+d_{2}\right)  ^{2}\left(
\begin{array}
[c]{c}%
\left[  X_{1}+Y_{1},X_{2}\right]  +\frac{1}{2}\left[  X_{2}+Y_{2}+\left[
X_{1},Y_{1}\right]  ,X_{1}\right]  +\\
\frac{1}{3}\left[  X_{1}+Y_{1},Y_{1},X_{1}\right]
\end{array}
\right)
\end{array}
\right\}
\end{align*}
and letting $\ast_{2}$ denote
\[
Y_{1}+\left(  d_{1}+d_{2}\right)  Y_{2}+\frac{1}{2}\left(  d_{1}+d_{2}\right)
^{2}\left(  Y_{3}-\frac{1}{2}\left[  Y_{1},Y_{2}\right]  \right)  \text{,}%
\]
we have
\begin{align}
& \exp\,\ast_{1}+d_{3}\ast_{2}\nonumber\\
& =\exp\,\ast_{1}.\exp\,d_{3}\left(  \ast_{2}-\frac{1}{2}\left[  \ast_{1}%
,\ast_{2}\right]  +\frac{1}{6}\left[  \ast_{1},\ast_{1},\ast_{2}\right]
\right) \label{4.3.5.1}%
\end{align}
by left logarithmic derivation. By the way, we have the following:
\[
\left[  \ast_{1},\ast_{2}\right]  =\left(  d_{1}+d_{2}\right)  \left[
X_{1},Y_{1}\right]  +\left(  d_{1}+d_{2}\right)  ^{2}\left(  \left[
X_{1}+Y_{1},Y_{2}\right]  +\frac{1}{2}\left[  X_{2}+Y_{2}+\left[  X_{1}%
,Y_{1}\right]  ,Y_{1}\right]  \right)
\]
\[
\left[  \ast_{1},\ast_{1},\ast_{2}\right]  =\left(  d_{1}+d_{2}\right)
^{2}\left[  X_{1}+Y_{1},X_{1},Y_{1}\right]
\]
\begin{align*}
& -\frac{1}{2}\left[  \ast_{1},\ast_{2}\right]  +\frac{1}{6}\left[  \ast
_{1},\ast_{1},\ast_{2}\right] \\
& =-\frac{1}{2}\left(  d_{1}+d_{2}\right)  \left[  X_{1},Y_{1}\right]  -\\
& \frac{1}{2}\left(  d_{1}+d_{2}\right)  ^{2}\left(  \left[  X_{1}+Y_{1}%
,Y_{2}\right]  +\frac{1}{2}\left[  X_{2}+Y_{2}+\left[  X_{1},Y_{1}\right]
,Y_{1}\right]  -\frac{1}{3}\left[  X_{1}+Y_{1},Y_{1},X_{1}\right]  \right)
\end{align*}
Therefore the desired formula follows.
\end{proof}

\begin{lemma}
\label{t4.3.6}
\begin{align*}
& \exp\,\left(  d_{1}+d_{2}\right)  \left(  X_{1}+Y_{1}\right)  +\frac{1}%
{2}\left(  d_{1}+d_{2}\right)  ^{2}\left(  X_{2}+Y_{2}+\left[  X_{1}%
,Y_{1}\right]  \right)  +\\
& d_{3}\left\{  X_{1}+\left(  d_{1}+d_{2}\right)  X_{2}+\left(  d_{1}%
+d_{2}\right)  ^{2}\left(  \frac{1}{2}X_{3}+\frac{1}{4}\left[  X_{1}%
,X_{2}\right]  \right)  \right\}  -\\
& d_{3}\left\{
\begin{array}
[c]{c}%
\frac{1}{2}\left(  d_{1}+d_{2}\right)  \left[  Y_{1},X_{1}\right]  +\\
\frac{1}{2}\left(  d_{1}+d_{2}\right)  ^{2}\left(
\begin{array}
[c]{c}%
\left[  X_{1}+Y_{1},X_{2}\right]  +\frac{1}{2}\left[  X_{2}+Y_{2}+\left[
X_{1},Y_{1}\right]  ,X_{1}\right]  +\\
\frac{1}{3}\left[  X_{1}+Y_{1},Y_{1},X_{1}\right]
\end{array}
\right)
\end{array}
\right\}  +\\
& d_{3}\left\{  Y_{1}+\left(  d_{1}+d_{2}\right)  Y_{2}+\frac{1}{2}\left(
d_{1}+d_{2}\right)  ^{2}\left(  Y_{3}-\frac{1}{2}\left[  Y_{1},Y_{2}\right]
\right)  \right\}  +\\
& d_{3}\left\{
\begin{array}
[c]{c}%
\frac{1}{2}\left(  d_{1}+d_{2}\right)  \left[  X_{1},Y_{1}\right]  +\\
\frac{1}{2}\left(  d_{1}+d_{2}\right)  ^{2}\left(
\begin{array}
[c]{c}%
\left[  X_{1}+Y_{1},Y_{2}\right]  +\frac{1}{2}\left[  X_{2}+Y_{2}+\left[
X_{1},Y_{1}\right]  ,Y_{1}\right]  -\\
\frac{1}{3}\left[  X_{1}+Y_{1},Y_{1},X_{1}\right]
\end{array}
\right)
\end{array}
\right\} \\
& =\exp\,\left(  d_{1}+d_{2}\right)  \left(  X_{1}+Y_{1}\right)  +\frac{1}%
{2}\left(  d_{1}+d_{2}\right)  ^{2}\left(  X_{2}+Y_{2}+\left[  X_{1}%
,Y_{1}\right]  \right)  +\\
& d_{3}\left\{  X_{1}+\left(  d_{1}+d_{2}\right)  X_{2}+\frac{1}{2}\left(
d_{1}+d_{2}\right)  ^{2}\left(  X_{3}+\frac{1}{2}\left[  X_{1},X_{2}\right]
\right)  \right\}  -\\
& d_{3}\left\{
\begin{array}
[c]{c}%
\frac{1}{2}\left(  d_{1}+d_{2}\right)  \left[  Y_{1},X_{1}\right]  +\\
\frac{1}{2}\left(  d_{1}+d_{2}\right)  ^{2}\left(
\begin{array}
[c]{c}%
\left[  X_{1}+Y_{1},X_{2}\right]  +\frac{1}{2}\left[  X_{2}+Y_{2}+\left[
X_{1},Y_{1}\right]  ,X_{1}\right]  +\\
\frac{1}{3}\left[  X_{1}+Y_{1},Y_{1},X_{1}\right]
\end{array}
\right)
\end{array}
\right\}  +\\
& d_{3}\left\{  Y_{1}+\left(  d_{1}+d_{2}\right)  Y_{2}+\frac{1}{2}\left(
d_{1}+d_{2}\right)  ^{2}\left(  Y_{3}-\frac{1}{2}\left[  Y_{1},Y_{2}\right]
\right)  \right\}  .\\
& \exp\,d_{3}\left\{
\begin{array}
[c]{c}%
\frac{1}{2}\left(  d_{1}+d_{2}\right)  \left[  X_{1},Y_{1}\right]  +\\
\frac{1}{2}\left(  d_{1}+d_{2}\right)  ^{2}\left(
\begin{array}
[c]{c}%
\left[  X_{1}+Y_{1},Y_{2}\right]  +\frac{1}{2}\left[  X_{2}+Y_{2}+\left[
X_{1},Y_{1}\right]  ,Y_{1}\right]  -\\
\frac{1}{3}\left[  X_{1}+Y_{1},Y_{1},X_{1}\right]  -\\
\frac{1}{2}\left[  X_{1}+Y_{1},X_{1},Y_{1}\right]
\end{array}
\right)
\end{array}
\right\}
\end{align*}

\end{lemma}

\begin{proof}
Letting $\ast_{1}$ denote
\begin{align*}
& \left(  d_{1}+d_{2}\right)  \left(  X_{1}+Y_{1}\right)  +\frac{1}{2}\left(
d_{1}+d_{2}\right)  ^{2}\left(  X_{2}+Y_{2}+\left[  X_{1},Y_{1}\right]
\right)  +\\
& d_{3}\left\{  X_{1}+\left(  d_{1}+d_{2}\right)  X_{2}+\frac{1}{2}\left(
d_{1}+d_{2}\right)  ^{2}\left(  X_{3}+\frac{1}{2}\left[  X_{1},X_{2}\right]
\right)  \right\}  -\\
& d_{3}\left\{
\begin{array}
[c]{c}%
\frac{1}{2}\left(  d_{1}+d_{2}\right)  \left[  Y_{1},X_{1}\right]  +\\
\frac{1}{2}\left(  d_{1}+d_{2}\right)  ^{2}\left(
\begin{array}
[c]{c}%
\left[  X_{1}+Y_{1},X_{2}\right]  +\frac{1}{2}\left[  X_{2}+Y_{2}+\left[
X_{1},Y_{1}\right]  ,X_{1}\right]  +\\
\frac{1}{3}\left[  X_{1}+Y_{1},Y_{1},X_{1}\right]
\end{array}
\right)
\end{array}
\right\}  +\\
& d_{3}\left\{  Y_{1}+\left(  d_{1}+d_{2}\right)  Y_{2}+\frac{1}{2}\left(
d_{1}+d_{2}\right)  ^{2}\left(  Y_{3}-\frac{1}{2}\left[  Y_{1},Y_{2}\right]
\right)  \right\}
\end{align*}
and letting $\ast_{2}$ denote
\begin{align*}
& \frac{1}{2}\left(  d_{1}+d_{2}\right)  \left[  X_{1},Y_{1}\right]  +\\
& \frac{1}{2}\left(  d_{1}+d_{2}\right)  ^{2}\left(  \left[  X_{1}+Y_{1}%
,Y_{2}\right]  +\frac{1}{2}\left[  X_{2}+Y_{2}+\left[  X_{1},Y_{1}\right]
,Y_{1}\right]  -\frac{1}{3}\left[  X_{1}+Y_{1},Y_{1},X_{1}\right]  \right)
\end{align*}
we have (\ref{4.3.2.1}). By the way, we have the following:
\[
\left[  \ast_{1},\ast_{2}\right]  =\frac{1}{2}\left(  d_{1}+d_{2}\right)
^{2}\left[  X_{1}+Y_{1},X_{1},Y_{1}\right]
\]
Therefore the desired formula follows at once.
\end{proof}

\section{\label{s6}Associativity}

From now on, $\mathfrak{g}$\ shall be an arbitrary Lie algebra not necessarily
coming from a Lie group as its Lie algebra. The principal objective in the
rest of this paper is to show that the spaces $\left(  \mathfrak{g}^{D_{n}%
}\right)  _{0}\,\left(  n=1,2,3\right)  $ are naturally endowed with Lie group
structures, which can be regarded as the Weil prolongations of a mythical
(i.e., not necessarily existing) Lie group whose Lie algebra is supposed to be
$\mathfrak{g}$. This section aims to demonstrate that the spaces $\left(
\mathfrak{g}^{D_{n}}\right)  _{0}\,\left(  n=1,2,3\right)  $ are naturally
endowed with \textit{associative} binary operations. First of all, let us
define binary operations on them.

\begin{definition}
\label{d6.1}Inspired by Theorems \ref{t4.1}, \ref{t4.2} and \ref{t4.3}, we
will define a binary operation on $\left(  \mathfrak{g}^{D_{n}}\right)
_{0}\,\left(  n=1,2,3\right)  $ as follows:

\begin{enumerate}
\item Given $dX_{1},dY_{1}\in\left(  \mathfrak{g}^{D_{1}}\right)  _{0}$, we
define
\[
dX_{1}.dY_{1}%
\]
to be
\[
d\left(  X_{1}+Y_{1}\right)
\]

\item Given $dX_{1}+\frac{1}{2}d^{2}X_{2},dY_{1}+\frac{1}{2}d^{2}Y_{2}%
\in\left(  \mathfrak{g}^{D_{2}}\right)  _{0}$, we define
\[
dX_{1}+\frac{1}{2}d^{2}X_{2}.dY_{1}+\frac{1}{2}d^{2}Y_{2}%
\]
to be
\[
d\left(  X_{1}+Y_{1}\right)  +\frac{1}{2}d^{2}\left(  X_{2}+Y_{2}+\left[
X_{1},Y_{1}\right]  \right)
\]

\item Given $dX_{1}+\frac{1}{2}d^{2}X_{2}+\frac{1}{6}d^{3}X_{3},dY_{1}%
+\frac{1}{2}d^{2}Y_{2}+\frac{1}{6}d^{3}Y_{3}\in\left(  \mathfrak{g}^{D_{3}%
}\right)  _{0}$, we define
\[
dX_{1}+\frac{1}{2}d^{2}X_{2}+\frac{1}{6}d^{3}X_{3}.dY_{1}+\frac{1}{2}%
d^{2}Y_{2}+\frac{1}{6}d^{3}Y_{3}%
\]
to be
\begin{align*}
& d\left(  X_{1}+Y_{1}\right)  +\frac{1}{2}d^{2}\left(  X_{2}+Y_{2}+\left[
X_{1},Y_{1}\right]  \right)  +\\
& \frac{1}{6}d^{3}\left\{  \left(  X_{3}+Y_{3}\right)  +\frac{3}{2}\left(
\left[  X_{1},Y_{2}\right]  +\left[  X_{2},Y_{1}\right]  \right)  +\frac{1}%
{2}\left[  X_{1}-Y_{1},X_{1},Y_{1}\right]  \right\}
\end{align*}

\end{enumerate}
\end{definition}

The principal objective in this section is to show that the above binary
operations are all associative. It should be obvious that

\begin{theorem}
\label{t6.1}
\[
\left(  dX_{1}.dY_{1}\right)  .dZ_{1}=dX_{1}.\left(  dY_{1}.dZ_{1}\right)
\]

\end{theorem}

\begin{theorem}
\label{t6.2}
\begin{align*}
& \left(  dX_{1}+\frac{1}{2}d^{2}X_{2}.dY_{1}+\frac{1}{2}d^{2}Y_{2}\right)
.dZ_{1}+\frac{1}{2}d^{2}Z_{2}\\
& =dX_{1}+\frac{1}{2}d^{2}X_{2}.\left(  dY_{1}+\frac{1}{2}d^{2}Y_{2}%
.dZ_{1}+\frac{1}{2}d^{2}Z_{2}\right)
\end{align*}

\end{theorem}

\begin{proof}
We have
\begin{align*}
& \left(  dX_{1}+\frac{1}{2}d^{2}X_{2}.dY_{1}+\frac{1}{2}d^{2}Y_{2}\right)
.dZ_{1}+\frac{1}{2}d^{2}Z_{2}\\
& =d\left(  X_{1}+Y_{1}\right)  +\frac{1}{2}d^{2}\left(  X_{2}+Y_{2}+\left[
X_{1},Y_{1}\right]  \right)  .dZ_{1}+\frac{1}{2}d^{2}Z_{2}\\
& =d\left(  X_{1}+Y_{1}+Z_{1}\right)  +\frac{1}{2}d^{2}\left(  X_{2}%
+Y_{2}+\left[  X_{1},Y_{1}\right]  +Z_{2}+\left[  X_{1}+Y_{1},Z_{1}\right]
\right) \\
& =d\left(  X_{1}+Y_{1}+Z_{1}\right)  +\frac{1}{2}d^{2}\left(  X_{2}%
+Y_{2}+Z_{2}+\left[  X_{1},Y_{1}\right]  +\left[  X_{1},Z_{1}\right]  +\left[
Y_{1},Z_{1}\right]  \right)
\end{align*}
on the one hand, while we have
\begin{align*}
& dX_{1}+\frac{1}{2}d^{2}X_{2}.\left(  dY_{1}+\frac{1}{2}d^{2}Y_{2}%
.dZ_{1}+\frac{1}{2}d^{2}Z_{2}\right) \\
& =dX_{1}+\frac{1}{2}d^{2}X_{2}.d\left(  Y_{1}+Z_{1}\right)  +\frac{1}{2}%
d^{2}\left(  Y_{2}+Z_{2}+\left[  Y_{1},Z_{1}\right]  \right) \\
& =d\left(  X_{1}+Y_{1}+Z_{1}\right)  +\frac{1}{2}d^{2}\left(  X_{2}%
+Y_{2}+Z_{2}+\left[  Y_{1},Z_{1}\right]  +\left[  X_{1},Y_{1}+Z_{1}\right]
\right) \\
& =d\left(  X_{1}+Y_{1}+Z_{1}\right)  +\frac{1}{2}d^{2}\left(  X_{2}%
+Y_{2}+Z_{2}+\left[  X_{1},Y_{1}\right]  +\left[  X_{1},Z_{1}\right]  +\left[
Y_{1},Z_{1}\right]  \right)
\end{align*}
on the other.
\end{proof}

\begin{theorem}
\label{t6.3}
\begin{align*}
& \left(  dX_{1}+\frac{1}{2}d^{2}X_{2}+\frac{1}{6}d^{3}X_{3}.dY_{1}+\frac
{1}{2}d^{2}Y_{2}+\frac{1}{6}d^{3}Y_{3}\right)  .dZ_{1}+\frac{1}{2}d^{2}%
Z_{2}+\frac{1}{6}d^{3}Z_{3}\\
& =dX_{1}+\frac{1}{2}d^{2}X_{2}+\frac{1}{6}d^{3}X_{3}.\left(  dY_{1}+\frac
{1}{2}d^{2}Y_{2}+\frac{1}{6}d^{3}Y_{3}.dZ_{1}+\frac{1}{2}d^{2}Z_{2}+\frac
{1}{6}d^{3}Z_{3}\right)
\end{align*}

\end{theorem}

\begin{proof}
We have
\begin{align*}
& \left(  dX_{1}+\frac{1}{2}d^{2}X_{2}+\frac{1}{6}d^{3}X_{3}.dY_{1}+\frac
{1}{2}d^{2}Y_{2}+\frac{1}{6}d^{3}Y_{3}\right)  .dZ_{1}+\frac{1}{2}d^{2}%
Z_{2}+\frac{1}{6}d^{3}Z_{3}\\
& =d\left(  X_{1}+Y_{1}\right)  +\frac{1}{2}d^{2}\left(  X_{2}+Y_{2}+\left[
X_{1},Y_{1}\right]  \right)  +\\
& \frac{1}{6}d^{3}\left\{  \left(  X_{3}+Y_{3}\right)  +\frac{3}{2}\left(
\left[  X_{1},Y_{2}\right]  +\left[  X_{2},Y_{1}\right]  \right)  +\frac{1}%
{2}\left[  X_{1}-Y_{1},X_{1},Y_{1}\right]  \right\}  .\\
& dZ_{1}+\frac{1}{2}d^{2}Z_{2}+\frac{1}{6}d^{3}Z_{3}\\
& =d\left(  X_{1}+Y_{1}+Z_{1}\right)  +\frac{1}{2}d^{2}\left(  X_{2}%
+Y_{2}+\left[  X_{1},Y_{1}\right]  +Z_{2}+\left[  X_{1}+Y_{1},Z_{1}\right]
\right)  +\\
& \frac{1}{6}d^{3}\left\{
\begin{array}
[c]{c}%
\left(  X_{3}+Y_{3}\right)  +\frac{3}{2}\left(  \left[  X_{1},Y_{2}\right]
+\left[  X_{2},Y_{1}\right]  \right)  +\frac{1}{2}\left[  X_{1}-Y_{1}%
,X_{1},Y_{1}\right]  +Z_{3}+\\
\frac{3}{2}\left(  [X_{1}+Y_{1},Z_{2}]+\left[  X_{2}+Y_{2}+\left[  X_{1}%
,Y_{1}\right]  ,Z_{1}\right]  \right)  +\\
\frac{1}{2}\left[  X_{1}+Y_{1}-Z_{1},\left[  X_{1}+Y_{1},Z_{1}\right]
\right]
\end{array}
\right\} \\
& =d\left(  X_{1}+Y_{1}+Z_{1}\right)  +\frac{1}{2}d^{2}\left(  X_{2}%
+Y_{2}+Z_{2}+\left[  X_{1},Y_{1}\right]  +\left[  X_{1},Z_{1}\right]  +\left[
Y_{1},Z_{1}\right]  \right)  +\\
& \frac{1}{6}d^{3}\left(
\begin{array}
[c]{c}%
X_{3}+Y_{3}+Z_{3}+\\
\frac{3}{2}\left(  \left[  X_{1},Y_{2}\right]  +[X_{1},Z_{2}]+[Y_{1}%
,Z_{2}]+\left[  X_{2},Y_{1}\right]  +[X_{2},Z_{1}]+[Y_{2},Z_{1}]\right)  +\\
\frac{1}{2}\left(
\begin{array}
[c]{c}%
\left[  X_{1},X_{1},Y_{1}\right]  +\left[  Y_{1},Y_{1},X_{1}\right]  +\\
\left[  X_{1},X_{1},Z_{1}\right]  +\left[  Z_{1},Z_{1},X_{1}\right]  +\\
\left[  Y_{1},Y_{1},Z_{1}\right]  +\left[  Z_{1},Z_{1},Y_{1}\right]
\end{array}
\right)  +\\
\frac{3}{2}\left[  \left[  X_{1},Y_{1}\right]  ,Z_{1}\right]  +\frac{1}%
{2}\left(  \left[  X_{1},Y_{1},Z_{1}\right]  +\left[  Y_{1},X_{1}%
,Z_{1}\right]  \right)
\end{array}
\right)
\end{align*}
on the one hand, while we have
\begin{align*}
& dX_{1}+\frac{1}{2}d^{2}X_{2}+\frac{1}{6}d^{3}X_{3}.\left(  dY_{1}+\frac
{1}{2}d^{2}Y_{2}+\frac{1}{6}d^{3}Y_{3}.dZ_{1}+\frac{1}{2}d^{2}Z_{2}+\frac
{1}{6}d^{3}Z_{3}\right) \\
& =dX_{1}+\frac{1}{2}d^{2}X_{2}+\frac{1}{6}d^{3}X_{3}.\\
& d\left(  Y_{1}+Z_{1}\right)  +\frac{1}{2}d^{2}\left(  Y_{2}+Z_{2}+\left[
Y_{1},Z_{1}\right]  \right)  +\\
& \frac{1}{6}d^{3}\left\{  \left(  Y_{3}+Z_{3}\right)  +\frac{3}{2}\left(
\left[  Y_{1},Z_{2}\right]  +\left[  Y_{2},Z_{1}\right]  \right)  +\frac{1}%
{2}\left[  Y_{1}-Z_{1},Y_{1},Z_{1}\right]  \right\} \\
& =d\left(  X_{1}+Y_{1}+Z_{1}\right)  +\frac{1}{2}d^{2}\left(  X_{2}%
+Y_{2}+Z_{2}+\left[  Y_{1},Z_{1}\right]  +\left[  X_{1},Y_{1}+Z_{1}\right]
\right)  +\\
& \frac{1}{6}d^{3}\left\{
\begin{array}
[c]{c}%
X_{3}+\left(  Y_{3}+Z_{3}\right)  +\frac{3}{2}\left(  \left[  Y_{1}%
,Z_{2}\right]  +\left[  Y_{2},Z_{1}\right]  \right)  +\frac{1}{2}\left[
Y_{1}-Z_{1},Y_{1},Z_{1}\right]  +\\
\frac{3}{2}\left(  \left[  X_{1},Y_{2}+Z_{2}+\left[  Y_{1},Z_{1}\right]
\right]  +\left[  X_{2},Y_{1}+Z_{1}\right]  \right)  +\\
\frac{1}{2}\left[  X_{1}-\left(  Y_{1}+Z_{1}\right)  ,X_{1},Y_{1}%
+Z_{1}\right]
\end{array}
\right\} \\
& =d\left(  X_{1}+Y_{1}+Z_{1}\right)  +\frac{1}{2}d^{2}\left(  X_{2}%
+Y_{2}+Z_{2}+\left[  X_{1},Y_{1}\right]  +\left[  X_{1},Z_{1}\right]  +\left[
Y_{1},Z_{1}\right]  \right)  +\\
& \frac{1}{6}d^{3}\left\{
\begin{array}
[c]{c}%
X_{3}+Y_{3}+Z_{3}+\\
\frac{3}{2}\left(  \left[  X_{1},Y_{2}\right]  +[X_{1},Z_{2}]+[Y_{1}%
,Z_{2}]+\left[  X_{2},Y_{1}\right]  +[X_{2},Z_{1}]+[Y_{2},Z_{1}]\right)  +\\
\frac{1}{2}\left(
\begin{array}
[c]{c}%
\left[  X_{1},X_{1},Y_{1}\right]  +\left[  Y_{1},Y_{1},X_{1}\right]  +\\
\left[  X_{1},X_{1},Z_{1}\right]  +\left[  Z_{1},Z_{1},X_{1}\right]  +\\
\left[  Y_{1},Y_{1},Z_{1}\right]  +\left[  Z_{1},Z_{1},Y_{1}\right]
\end{array}
\right)  +\\
\frac{3}{2}\left[  X_{1},Y_{1},Z_{1}\right]  -\frac{1}{2}\left(  \left[
Y_{1},X_{1},Z_{1}\right]  +\left[  Z_{1},X_{1},Y_{1}\right]  \right)
\end{array}
\right\}
\end{align*}
on the other hand. Therefore we are well done by the following lemma.
\end{proof}

\begin{lemma}
\label{t6.3.1}We have
\begin{align*}
& \frac{3}{2}\left[  \left[  X_{1},Y_{1}\right]  ,Z_{1}\right]  +\frac{1}%
{2}\left(  \left[  X_{1},Y_{1},Z_{1}\right]  +\left[  Y_{1},X_{1}%
,Z_{1}\right]  \right) \\
& =\frac{3}{2}\left[  X_{1},Y_{1},Z_{1}\right]  -\frac{1}{2}\left(  \left[
Y_{1},X_{1},Z_{1}\right]  +\left[  Z_{1},X_{1},Y_{1}\right]  \right)
\end{align*}

\end{lemma}

\begin{proof}
As is expected, this follows easily from the Jacobi identity. We have
\begin{align*}
& \left\{  \frac{3}{2}\left[  \left[  X_{1},Y_{1}\right]  ,Z_{1}\right]
+\frac{1}{2}\left(  \left[  X_{1},Y_{1},Z_{1}\right]  +\left[  Y_{1}%
,X_{1},Z_{1}\right]  \right)  \right\}  -\\
& \left\{  \frac{3}{2}\left[  X_{1},Y_{1},Z_{1}\right]  -\frac{1}{2}\left(
\left[  Y_{1},X_{1},Z_{1}\right]  +\left[  Z_{1},X_{1},Y_{1}\right]  \right)
\right\} \\
& =\frac{3}{2}\left(  \left[  \left[  X_{1},Y_{1}\right]  ,Z_{1}\right]
-\left[  X_{1},Y_{1},Z_{1}\right]  \right)  +\\
& \frac{1}{2}\left(  \left[  X_{1},Y_{1},Z_{1}\right]  +\left[  Z_{1}%
,X_{1},Y_{1}\right]  \right)  +\left[  Y_{1},X_{1},Z_{1}\right] \\
& =-\frac{3}{2}\left[  \left[  Z_{1},X_{1}\right]  ,Y_{1}\right]  -\frac{1}%
{2}\left[  Y_{1},Z_{1},X_{1}\right]  +\left[  Y_{1},X_{1},Z_{1}\right] \\
& \left)
\begin{array}
[c]{c}%
\left[  \left[  X_{1},Y_{1}\right]  ,Z_{1}\right]  -\left[  X_{1},Y_{1}%
,Z_{1}\right]  =-\left[  \left[  Z_{1},X_{1}\right]  ,Y_{1}\right]  \text{ }\\
\text{and}\\
\left[  X_{1},Y_{1},Z_{1}\right]  +\left[  Z_{1},X_{1},Y_{1}\right]  =-\left[
Y_{1},Z_{1},X_{1}\right]  \text{ }\\
\text{by the Jacobi identity}%
\end{array}
\right( \\
& =0
\end{align*}

\end{proof}

\section{\label{s7}From Lie Algebras to Lie Groups}

\begin{theorem}
\label{t7.0}The spaces $\left(  \mathfrak{g}^{D_{n}}\right)  _{0}$\ $\left(
n=1,2,3\right)  $ are Lie groups with respect to the binary operations in
Definition \ref{d6.1}.
\end{theorem}

\begin{proof}
The microlinearity of $\left(  \mathfrak{g}^{D_{n}}\right)  _{0}$\ follows
from that of $\mathfrak{g}$. We have already seen that the binary operations
are associative. In order to let us done, we have only to note, in case of
$n=3$ by way of example, that $0$ is the unit element, while the inverse
element of
\[
dX_{1}+\frac{1}{2}d^{2}X_{2}+\frac{1}{6}d^{3}X_{3}\in\left(  \mathfrak{g}%
^{D_{3}}\right)  _{0}%
\]
is
\[
d\left(  -X_{1}\right)  +\frac{1}{2}d^{2}\left(  -X_{2}\right)  +\frac{1}%
{6}d^{3}\left(  -X_{3}\right)
\]

\end{proof}

In order to be sure whether the Lie group structure on $\left(  \mathfrak{g}%
^{D_{n}}\right)  _{0}$\ in Theorem \ref{t7.0}\ is indeed that of an
appropriate Weil prolongation $\left(  G^{D_{n}}\right)  _{1}$\ of a
\textit{mythical} Lie group $G$\ whose Lie algebra is supposed to be
$\mathfrak{g}$, we need to see its Lie algebra in computation.

\begin{theorem}
\label{t7.1}With $d,e_{1},e_{2}\in D$, we have
\[
\left[  de_{1}X_{1},de_{2}Y_{1}\right]  =0\text{,}%
\]
as is expected in Corollary \ref{t3.1.2}.
\end{theorem}

\begin{theorem}
\label{t7.2}With $d\in D_{2}$ and $e_{1},e_{2}\in D$, we have
\begin{align*}
& \left[  de_{1}X_{1}+\frac{1}{2}d^{2}e_{1}X_{2},de_{2}Y_{1}+\frac{1}{2}%
d^{2}e_{2}Y_{2}\right] \\
& =d^{2}e\left[  X_{1},Y_{1}\right]  \text{,}%
\end{align*}
as is expected in Corollary \ref{t3.1.2}.
\end{theorem}

\begin{proof}
We have
\begin{align*}
& de_{1}X_{1}+\frac{1}{2}d^{2}e_{1}X_{2}.de_{2}Y_{1}+\frac{1}{2}d^{2}%
e_{2}Y_{2}.d\left(  -e_{1}\right)  X_{1}+\\
& \frac{1}{2}d^{2}\left(  -e_{2}\right)  X_{2}.d\left(  -e_{1}\right)
Y_{1}+\frac{1}{2}d^{2}\left(  -e_{2}\right)  Y_{2}\\
& =d\left(  e_{1}X_{1}+e_{2}Y_{1}\right)  +\frac{1}{2}d^{2}\left(  e_{1}%
X_{2}+e_{2}Y_{2}+e_{1}e_{2}\left[  X_{1},Y_{1}\right]  \right)  .\\
& d\left(  -e_{1}X_{1}-e_{2}Y_{1}\right)  +\frac{1}{2}d^{2}\left(  -e_{1}%
X_{2}-e_{2}Y_{2}+e_{1}e_{2}\left[  X_{1},Y_{1}\right]  \right) \\
& =\frac{1}{2}d^{2}2e_{1}e_{2}\left[  X_{1},Y_{1}\right]
\end{align*}

\end{proof}

\begin{theorem}
\label{t7.3}With $d\in D_{3}$ and $e_{1},e_{2}\in D$, we have
\begin{align*}
& \left[  de_{1}X_{1}+\frac{1}{2}d^{2}e_{1}X_{2}+\frac{1}{6}d^{3}e_{1}%
X_{3},de_{2}Y_{1}+\frac{1}{2}d^{2}e_{2}Y_{2}+\frac{1}{6}d^{3}e_{2}Y_{3}\right]
\\
& =d^{2}e\left[  X_{1},Y_{1}\right]  +\frac{1}{2}d^{3}e\left(  \left[
X_{1},Y_{2}\right]  +\left[  X_{2},Y_{1}\right]  \right)  \text{,}%
\end{align*}
as is expected in Corollary \ref{t3.1.2}.
\end{theorem}

\begin{proof}%
\begin{align*}
& de_{1}X_{1}+\frac{1}{2}d^{2}e_{1}X_{2}+\frac{1}{6}d^{3}e_{1}X_{3}%
.de_{2}Y_{1}+\frac{1}{2}d^{2}e_{2}Y_{2}+\frac{1}{6}d^{3}e_{2}Y_{3}.\\
& d\left(  -e_{1}\right)  X_{1}+\frac{1}{2}d^{2}\left(  -e_{1}\right)
X_{2}+\frac{1}{6}d^{3}\left(  -e_{1}\right)  X_{3}.d\left(  -e_{2}\right)
Y_{1}+\frac{1}{2}d^{2}\left(  -e_{2}\right)  Y_{2}+\frac{1}{6}d^{3}\left(
-e_{2}\right)  Y_{3}\\
& =d\left(  e_{1}X_{1}+e_{2}Y_{1}\right)  +\frac{1}{2}d^{2}\left(  e_{1}%
X_{2}+e_{2}Y_{2}+e_{1}e_{2}\left[  X_{1},Y_{1}\right]  \right)  +\\
& \frac{1}{6}d^{3}\left\{  \left(  e_{1}X_{3}+e_{2}Y_{3}\right)  +\frac{3}%
{2}e_{1}e_{2}\left(  \left[  X_{1},Y_{2}\right]  +\left[  X_{2},Y_{1}\right]
\right)  \right\}  .\\
& d\left(  \left(  -e_{1}\right)  X_{1}+\left(  -e_{2}\right)  Y_{1}\right)
+\frac{1}{2}d^{2}\left(  \left(  -e_{1}\right)  X_{2}+\left(  -e_{2}\right)
Y_{2}+e_{1}e_{2}\left[  X_{1},Y_{1}\right]  \right)  +\\
& \frac{1}{6}d^{3}\left\{  \left(  \left(  -e_{1}\right)  X_{3}+\left(
-e_{2}\right)  Y_{3}\right)  +\frac{3}{2}e_{1}e_{2}\left(  \left[  X_{1}%
,Y_{2}\right]  +\left[  X_{2},Y_{1}\right]  \right)  \right\} \\
& =d^{2}e_{1}e_{2}\left[  X_{1},Y_{1}\right]  +\frac{1}{2}d^{3}e_{1}%
e_{2}\left(  \left[  X_{1},Y_{2}\right]  +\left[  X_{2},Y_{1}\right]  \right)
\end{align*}

\end{proof}

\end{document}